\documentclass[11pt]{amsart}

\newtheorem{theorem}{Theorem}[section]
\newtheorem{lemma}[theorem]{Lemma}

\newtheorem{proposition}[theorem]{Proposition}

\theoremstyle{definition}

\usepackage{tikz-cd}
\usepackage{mathrsfs}
\usepackage[all,cmtip]{xy}
\usepackage{amsthm}

\newcommand{\C}{\mathbb{C}}

\newcommand{\D}{\mathbb{D}}

\newcommand{\sD}{\mathscr{D}}
\newcommand{\sDk}{\mathscr{D}_K}
\newcommand{\uz}{\underline{z}}
\newcommand{\uzi}{\underline{z}_I}
\newcommand{\uw}{\underline{w}}
\newcommand{\sDo}{\mathscr{D}_\Omega}
\newcommand{\Cn}{\mathbb{C}^n}
\newcommand{\Tw}{T_{\underline{w}}}

\newcommand{\co}{c_\Omega ^*}

\theoremstyle{remark}

\newtheorem*{definition*}{\textbf{Definition}}  
\newtheorem*{example*}{\textbf{Example}}  

\numberwithin{equation}{section}

\begin{document}

\title{On interpolation in Carath\'eodory hyperbolic domains}

\author{Anindya Biswas}
\address{Department of Mathematics and Statistics, Masaryk University, Brno}
\address{ORCID: 0000-0002-7805-9446}
\email{anindyab132652@gmail.com}
\thanks{}





\date{}

\dedicatory{}

\begin{abstract}
We study the relation between Pick bodies on Carath\'eodory hyperbolic domains and contractions on finite dimensional Hilbert spaces. We give a condition sufficient to realize Pick bodies on Carath\'eodory hyperbolic domains as a Pick body on the open unit disc.

\end{abstract}

\maketitle

\section{Introduction and statements of the results} 

Suppose that $\Omega$ is a Carath\'eodory hyperbolic domain in the complex space $\C^m$ and $z_1,\ldots, z_n$ are mutually distinct points in $\Omega$. The object of our study is the \textit{Pick body}  (also known as the \textit{interpolation body}) which is given by
\begin{align*}
	\sDo(z_1,\ldots,z_n)=\{(f(z_1),\ldots,f(z_n)):f\in \mathcal{O}(\Omega,\overline{\D})\},
\end{align*}
where $\mathcal{O}(\Omega,\overline{\D})$ is the space of all holomorphic functions from $\Omega$ to the closure of the open unit disc $\D$. We will often abbreviate $n$-tuples $(x_1,\ldots,x_n)$ as $\underline{x}$ when there is no scope of confusion. The study of Pick bodies, was initiated by Cole, Lewis, and Wermer (\cite{CLW1992}, \cite{CLW1993}, \cite{CW1993}, \cite{CW1996}, \cite{CW1997}) to understand the Pick interpolation problem which asks for a necessary and sufficient condition for a given interpolation problem $\Omega\ni z_j\mapsto w_j\in \overline{\D},j=1,\ldots,n,$ to have a solution in $\mathcal{O}(\Omega,\overline{\D})$. When $\Omega=\D$, the problem has a well understood solution: An interpolation problem $\D\ni \lambda_j\mapsto w_j\in \overline{\D}$ has a solution in $\mathcal{O}(\D,\overline{\D})$ if and only if the matrix
\begin{align}\label{PickMatrix}
	\Bigg(\frac{1-w_i\overline{w_j}}{1-\lambda_i\overline{\lambda_j}}\Bigg)_{i,j}
\end{align}
is positive semidefinite. Moreover, the solution is unique if and only if the rank of the matrix (\ref{PickMatrix}) is less than $n$, and in that case the solution a finite Blaschke product whose degree equals the rank (\cite{A-M},	\cite{Nevanlinna1919}, \cite{Pick1916}, for a simpler proof see \cite{Marshall}). Thus 
\begin{align*}
	\sD_\D (\underline{\lambda})=\Big\{(w_1,\ldots,w_n)\in \C^n:\Big(\frac{1-w_i\overline{w_j}}{1-\lambda_i\overline{\lambda_j}}\Big)_{i,j}\geq \mathbf{0}\Big\}.
\end{align*}

Since $\lambda_1,\ldots,\lambda_n$ are mutually distinct, the matrix $\Big(\frac{1}{1-\lambda_i\overline{\lambda_j}}\Big)_{i,j}$ is positive definite and we can find a map $k:\{\lambda_1,\ldots, \lambda_n\}\rightarrow \Cn$ such that $\langle k(\lambda_j),k(\lambda_i)\rangle=\frac{1}{1-\lambda_i\overline{\lambda_j}}$. This gives rise to an inner product $\langle .,.\rangle_{\underline{\lambda}}$ on $\Cn$ which is given by 
\begin{align*}
	\langle (t_1,\ldots,t_n), (s_1,\ldots, s_n)\rangle_{\underline{\lambda}}=\sum_{i,j=1}^{n} t_i\overline{s_j}\langle k(\lambda_i),k(\lambda_j)\rangle.
\end{align*}
For $\underline{w}=(w_1,\ldots,w_n)\in \Cn$, if we define $T_{\underline{w}}(k(\lambda_j))=\overline{w_j}k(\lambda_j)$, then we find that $\sD_\D (\underline{\lambda})$ can be identified with $\{\underline{w}\in \Cn:||T_{\underline{w}}||\leq 1\}$, that is, $\sD_\D (\underline{\lambda})$ is the closed unit ball of $\Cn$, equipped with some norm (in our case an operator norm, see \cite{Sarason67}). A generalization of this set up is the following: Consider a positive definite $n\times n$ matrix $K=(K(i,j))_{1\leq i,j\leq n}$ and define 
\begin{align}
	\sDk=\big\{\underline{w}\in \Cn:\big((1-w_i\overline{w_j}) K(i,j)\big)\geq \mathbf{0}\big\}.
\end{align}
For the case of $\D$, we see that there is a single positive definite matrix $K$ satisfying $\sD_\D (\underline{\lambda})=\sDk$. Sometimes, we express $K$ as a $\C$-valued function defined on $\{z_1,\ldots,z_n\}\times \{z_1,\ldots,z_n\}$ to emphasize the involvement of the points $\{z_1,\ldots,z_n\}$ and in that case we will write $\sDk(\underline{z})$ instead of $\sDk$.

When $\Omega$ is the bidisc $\D^2$, the article \cite{CW1996} explores the question of finding a single matrix $K$ so that $\sD_{\D^2}(\underline{z})$ can be expressed as $\sDk$. In this article, we extend some of those results for any Carath\'eodory hyperbolic domain. Our first result in this direction is the following.
\begin{theorem}\label{PickBodyAsIntersection}
For any Carath\'eodory hyperbolic domain $\Omega$ and $n$ distinct points $z_1,\ldots, z_n$ in $\Omega$, there is a collection $\mathcal{K}_\Omega$ of positive definite $n\times n$ matrices such that 
\begin{align}\label{IntersectionOfKernelBalls}
	\sD_\Omega(\underline{z})=\bigcap_{K\in \mathcal{K}_\Omega} \sD_K.
\end{align}
\end{theorem}
 Given $z_1,\ldots,z_n\in \Omega$ and a map $\phi\in \mathcal{O}(\D,\Omega)$ sending $\lambda_j$ to $z_j$, it is known that $\sDo(\uz)=\sD_\D (\underline{\lambda})$ if and only if $\phi$ is a Carath\'eodory geodesic (\cite{Biswas2025}). Therefore, having $\sDo(\uz)=\sDk$ for a single matrix $K$ is a necessary condition for the existence of a Carath\'eodory geodesic containing all $z_j$. In the next result, we give a condition under which $\sDo(z_1,z_2,z_3)$ is expressible as $\sD_\D (\lambda_1,\lambda_2,\lambda_3)$ (cf. Theorem 3.  and Lemma 4.1. in \cite{CW1996}).

 \begin{theorem}\label{3PointCharac}
 	Suppose that $z_1,z_2,z_3$ are three distinct points in a Carath\'eodory hyperbolic domain $\Omega$ and $\underline{\alpha}=(\alpha_1,\alpha_2,\alpha_3)\in \partial\sDo(\uz)\cap\D^3$ is a point such that the following hold:
 	\begin{enumerate}
 		\item There is $3\times 3$ positiive definite matrix $K$ such that $\sDk=\sDo(\uz)$.
 		\item For at least two pairs $(i,j),1\leq i<j\leq 3$, $(\alpha_i,\alpha_j)\in \partial\sDo(z_i,z_j)$.
 	\end{enumerate}
 	Then $\alpha_1,\alpha_2$ and $\alpha_3$ are distinct and $\sDo(\uz)=\sD_\D (\underline{\alpha})$.
 \end{theorem}

For $\D^2$, the coordinate functions play a fundamental role in the study of $\mathcal{O}(\D^2,\overline{\D})$. Apart from the fact that these functions are essential for the Agler decomposition (\cite{A-M-Paper}, \cite{A-M}), a building block for the study of $\mathcal{O}(\D^2,\overline{\D})$, they are the unique Carath\'eodory extremals and contains sufficient data to identify $\D^2$ (\cite{AglLykYouCharac}). The significance of this extremality  of certain functions with respect to the points $z_j$ is implicit in Cole and Wermer's work on $\D^2$ (\cite{CW1996}). Also the fact that for $\D^2$ there are only two such functions has an important role in that development (cf. the proof of Theorem 3. in \cite{CW1996}). But, for general domains, neither such convenient Agler decomposition exists, nor the Carath\'eodory extremals are well understood yet. However, it is possible to find meaningful results in the context of Pick bodies using the Carath\'eodory pseudodistance, the definition of which and a generalization of it are the following.
\begin{definition*}
	Suppose that $z_1,\ldots,z_n$ are points in a domain $\Omega$. 
	\begin{enumerate}
		\item The Carath\'eodory pseudodistance between $z_1$ and $z_2$ is defined by
		\begin{align*}
			c_\Omega ^* (z_1,z_2)=sup\{m(f(z_1),f(z_2))=\Big|\frac{f(z_1)-f(z_2)}{1-\overline{f(z_1)}f(z_2)}\Big|:f\in \mathcal{O}(\Omega,\D)\}.
		\end{align*}
		\item The generalized Carath\'eodory function (also called the generalized M\"obius function, see \cite{J-P-Invariant}) is defined by
		\begin{align*}
			c_\Omega ^*(z_1;z_2,\ldots,z_n)=sup\{|f(z_1)|:f\in \mathcal{O}(\Omega,\D), f(z_2)=\ldots=f(z_n)=0\}.
		\end{align*}
	\end{enumerate}
	 
\end{definition*}
Before stating the next results, we need a few more definitions.
\begin{definition*}
	Suppose that $z_1,\ldots,z_n$ are mutually distinct points in a domain $\Omega$ and $K:\{z_1,\ldots,z_n\}\times \{z_1,\ldots,z_n\}\rightarrow \C$ is a function such that $\big(K(z_i,z_j)\big)_{1\leq i,j\leq n}$ is a positive definite matrix. Let $I=(i_1,\ldots,i_k)$ be a $k$-tuple of integers such that $1\leq i_1<\ldots<i_k\leq n$. 
	\begin{enumerate}
		\item By $\pi_I:\mathbb{C}^n\rightarrow\mathbb{C}^k$ we denote the mapping $\underline{z}\mapsto \underline{z} _I=(z_{i_1},\ldots,z_{i_k})$.
		\item For $\underline{z}\in \Omega^n$, we define $\sD_\Omega(\underline{z}_I)=\{(f(z_{i_1}),\ldots,f(z_{i_k})):f\in \mathcal{O}(\Omega,\overline{\D})\}$.
		\item The set $\{(w_1,\ldots,w_k)\in \mathbb{C}^k:\Big((1-w_{l}\overline{w_{m}})K(z_{i_l},z_{i_m})\Big)\geq \mathbf{0}\}$ will be denoted by $\sD_K(z_{i_1},\ldots,z_{i_k})=\sD_K(\underline{z}_I)$.
	\end{enumerate}
\end{definition*}
It is easy to verify that when $\Omega$ is Carath\'eodory hyperbolic, $\pi_I(\sD_\Omega(\underline{z}))=\sD_\Omega(\underline{z}_I)$ for any $I$ as above. One way to see this is to observe that a point of $\partial\sD_\Omega(\underline{z}_I)$ (the boundary of a set $X$ will be denoted by $\partial X$) can not come from the components of an interior point of $\sD_\Omega(\underline{z})$. Note that if $\sDo(\uz)=\sD_\D (\underline{\lambda})$, then for any $k$-tuple $I$, $\sDo(\uz_I)=\sD_\D (\underline{\lambda}_I)$. Observing this property, we now define the following.
\begin{definition*}
	A function $K:\{z_1,\ldots,z_n\}\times \{z_1,\ldots,z_n\}\rightarrow \mathbb{C}$ is said to be an $n$-extremal kernel (or, an extremal kernel if $n$ is clear from the context) if the matrix $(K(z_i,z_j))_{1\leq i,j\leq n}$
	is positive definite and for each $k$-tuple $I$ and each $\underline{w}^\prime\in \partial\sD_K(\underline{z}_I)$, there is a $\underline{w}\in\partial\sD_K(\underline{z})$ such that $\pi_I(\underline{w})=\underline{w}^\prime$. 
\end{definition*}
\begin{example*}
	\begin{enumerate}
		\item Suppose that $(K(i,j))_{1\leq i,j\leq n}$ is a diagonal $n\times n$ matrix with positive diagonal entries. It is easy to see that $\sD_K$ is the closed unit polydisc in $\mathbb{C}^n$. Since a point lies on the boundary of a polydisc if and only if at least one of its component has modulus one, we see that $K$ is an extremal kernel.
		\item  Let $\lambda_1,\ldots, \lambda_n$ be $n$ distinct points in the open unit disc $\mathbb{D}$. Then the function $K$ defined on $\{\lambda_1,\ldots, \lambda_n\}\times \{\lambda_1,\ldots, \lambda_n\}$ by $K(\lambda_i,\lambda_j)=\frac{1}{1-\lambda_i\overline{\lambda_j}}$ gives an extremal kernel. If we are given a $k$-tuple $I$, then for any $\underline{w}^\prime\in \partial\sD_K(\underline{z}_I)$, there is a unique Blaschke product of degree at most $k-1$ that interpolates $\underline{\lambda}_I\mapsto\underline{w}^\prime$ (see \cite{Biswas2025}). The values of this Blaschke product at other $\lambda_j$ give us the corresponding point which is projected by $\pi_I$ onto $\underline{w}^\prime$ 
	\end{enumerate}
\end{example*}
We later show that if $\sDk(\uz)$ is a Pick body $\sDo(\uz)$, then the entries of $K$ are determined up to unimodular constants. Now, with the notion of extremal kernels, we have two theorems. The first one gives an estimate for the generalized Carath\'eodory function.
\begin{theorem}\label{BoundOnGeneralizedCara}
	Let $z_1,z_2,z_3$ be three distinct points in a Carath\'eodory hyperbolic domain $\Omega$, and $K$ a $3$-extremal kernel defined on $\{z_1,z_2,z_3\}$ satisfying $\sDo(\uz)=\sDk(\uz)$. If $\underline{\alpha}=(\alpha_1,\alpha_2,\alpha_3)\in \partial \sDo(\uz)\cap\D^3$ satisfies $m(\alpha_1,\alpha_2)=\co(z_1,z_2)$, then 
	\begin{align*}
		\co(z_3;z_1,z_2)\geq max\Biggl\{\sqrt{\frac{\co(z_j,z_3)^2 -m(\alpha_j,\alpha_3)^2}{1-m(\alpha_j,\alpha_3)^2}}:j=1,2\Biggr\}.
	\end{align*}
\end{theorem}
The following theorem gives a sufficient condition for a $\sDo(\uz)$ to be expressible as $\sD_\D (\underline{\lambda})$.
\begin{theorem}\label{nPointTheorem}
	Suppose that $z_1,\ldots,z_n$ are mutually distinct points in a Carath\'eodory hyperbolic domain $\Omega$ and $K$ is an extremal kernel on $\{z_1,\ldots,z_n\}$ with $K(z_i,z_i)=1$ for all $i$. If $\sDo(\uz)=\sDk(\uz)$ and there is an $\underline{\alpha}=(\alpha_1,\ldots,\alpha_n)\in \sDo(\uz)$ such that $m(\alpha_1,\alpha_j)=\co(z_1,z_j)$ for $j=2,\ldots, n$, then $\sDo(\uz)=\sD_\D(\underline{\alpha})$.
\end{theorem}

\section{Proofs of the results}
First we prove Theorem \ref{PickBodyAsIntersection}. Before going into the proof we recall the notion of Schur-Agler class generated by a collection of test function (\cite{B-H}, \cite{BhatBisCha}, \cite{D-M}). Suppose we are given a collection $\Psi$ of functions in $\mathcal{O}(\Omega,\D)$, and we consider kernels $\Gamma:\Omega\times\Omega\rightarrow \C$ such that $$\big((1-\psi(z),\overline{\psi(w)})\Gamma(z,w)\big)_{z,w\in \Omega}\geq \mathbf{0}\,\,\text{for all}\,\,\psi\in \Psi.$$
The set of all such $\Gamma$ is called the set of admissible kernels and is denoted by $\mathcal{K}_\Psi$.
The $\Psi$-Schur-Agler class is the collection of all functions $f:\Omega\rightarrow\C$ such that 
$$\big((1-f(z),\overline{f(w)})\Gamma(z,w)\big)_{z,w\in \Omega}\geq \mathbf{0}\,\,\text{for all}\,\,\Gamma\in \mathcal{K}_\Psi.$$

\begin{proof}[\textbf{Proof of Theorem \ref{PickBodyAsIntersection}}]

		We fix an arbitrary point $z_0\in \Omega$ different from $z_j$, $1\leq j\leq n$, and consider $\Psi=\{f\in \mathcal{O}(\Omega,\D):f(z_0)=0\}$. By Proposition $2.5$ in \cite{BiswasSchur}, the Schur-Agler class generated by $\Psi$ is precisely the space $\mathcal{O}(\Omega,\overline{\D})$. Let the set of all $\Psi$-admissible kernels be denoted by $\mathcal{K}_\Psi$. For $\Gamma\in \mathcal{K}_\Psi$, if $K_\Gamma=\Gamma|_{\{z_1,\ldots,z_n\}\times \{z_1,\ldots,z_n\}}$ is positive semidefinite, we consider the collection $\{K_\Gamma +\frac{1}{m}I_n:m\geq 1\}$. Clearly each of these matrices is positive definite. If $K_\Gamma$ is positive definite, we consider the singleton set $\{K_\Gamma\}$. Let $\mathcal{K}_\Omega$ be the union of the collections we just constructed. 
		
		Now we show that $\mathcal{K}_\Omega$ satisfies (\ref{IntersectionOfKernelBalls}). Let $(w_1,\ldots,w_n)\in \bigcap_{K\in \mathcal{K}_\Omega} \sD_K$ and $\Gamma\in \mathcal{K}_\Psi$. If $K_\Gamma$ is positive definite, then it is in $K_\Omega$ and we have
		\begin{align*}
			\big((1-w_i\overline{w_j})K_\Gamma (z_i,z_j)\big)=\big((1-w_i\overline{w_j})\Gamma (z_i,z_j)\big)\geq \mathbf{0}.
		\end{align*}
		If $K_\Gamma$ is positive semidefinite, then $K_\Gamma +\frac{1}{m}I_n\in \mathcal{K}_\Omega$ for all $m\geq 1$. Hence, for all $m$ we have 
		\begin{align*}
			\big((1-w_i\overline{w_j})\Gamma (z_i,z_j)\big) +\frac{1}{m}diag(1-|w_1|^2,\ldots,1-|w_n|^2)\geq \mathbf{0}.
		\end{align*}
		So we conclude $\big((1-w_i\overline{w_j})\Gamma (z_i,z_j)\big)\geq \mathbf{0}$ for all $\Gamma\in \mathcal{K}_\Psi$ and therefore, there is an element $f$ in the $\Psi$-Schur-Agler class, that is, in $\mathcal{O}(\Omega, \overline{\D})$ such that $f(z_j)=w_j, 1\leq j\leq n$. This shows that $(w_1,\ldots,w_n)\in \sD_\Omega(\underline{z})$.
		
		Conversely, if $(w_1,\ldots,w_n)\in \sD_\Omega(\underline{z})$, then we find an $f\in \mathcal{O}(\Omega,\overline{\D})$ sending $z_j$ to $w_j$. For any $\Gamma\in \mathcal{K}_\Psi$, $\big((1-f(z')\overline{f(z'')})\Gamma(z',z'')\big)_{z',z''\in \Omega}$ is a positive kernel. In particular, $\big((1-w_i\overline{w_j})\Gamma (z_i,z_j)\big) +\frac{1}{m}diag(1-|w_1|^2,\ldots,1-|w_n|^2)\geq \mathbf{0}$ irrespective of whether $\Gamma|_{\{z_1,\ldots,z_n\}\times \{z_1,\ldots,z_n\}}$ is positive definite or not. By our construction of $\mathcal{K}_\Omega$, $(w_1,\ldots, w_n)\in \sD_K$ for all $K\in \mathcal{K}_\Omega$.
		
		This completes the proof.
	\end{proof}
	
	Suppose that we are given a compact convex subset $D$ of $\Cn$ with nonempty interior and $\uw\in \partial D$. If the supporting hyperplane at $\uw$ is not unique, we call $\uw$ a singular point. When there is a unique supporting hyperplane, we call call $\uw$ a smooth point (\cite{CW1996}). The following proposition and Theorem \ref{3PointCharac} are similar to the corresponding results in \cite{CW1996}. Since our definitions slightly differ from theirs, we have included the  proofs here. The proofs follow almost the same arguments that are in \cite{CW1996}.

	\begin{proposition}\label{SingularElement}
		Suppose that $K$ is a $3\times 3$ positive definite matrix, $\sDk=\{\uw\in \C^3:\big((1-w_i\overline{w_i})K_{i,j}\big)\geq \mathbf{0}\}$, and $\underline{\alpha}=(\alpha_1,\alpha_2,\alpha_3)\in \partial\sDk\cap\D^3$ is a singular point on $\partial\sDk$. Then $\sDk=\sD_\D (\underline{\alpha})$.
	\end{proposition}
	\begin{proof}
		Since $K$ is a positive definite matrix there exist three linearly independent elements $k_1,k_2,k_3\in \C^3$ such that $K_{i,j}=\langle k_j,k_i\rangle$. For $\uw\in\C^3$, we define $T_{\uw}:\C^3\rightarrow\C^3$ by $\Tw(k_j)=\overline{w_j}k_j$. Then $\sDk=\{\uw\in \C^3:I_3-\Tw^* \Tw\geq \mathbf{0}\}=\{\uw\in \C^3:||\Tw||\leq 1\}$. Since $\underline{\alpha}\in \partial\sDk$, $||T_{\underline{\alpha}}||=1$, and hence $rank(I_3-\Tw^* \Tw)$ is one of $0,1$ and $2$.

		Since $\underline{\alpha}\in \D^3$, the rank is not $0$. Suppose that the rank is $2$. If the eigenvalues of $T_{\underline{\alpha}}^*T_{\underline{\alpha}}$ are $\lambda_1,\lambda_2,\lambda_3$, we may arrange them as $0\leq \lambda_1\leq\lambda_2< \lambda_3=1$ (using the spectral theorem and the fact that our operator is a positive one). Write \begin{align*}
			F(\lambda,\uw)&=det\big(\lambda I_3-\Tw^*\Tw\big)\\
			&=\lambda^3 -a(\uw)\lambda^2-b(\uw)\lambda-c(\uw)
		\end{align*}
		where $a,b$ and $c$ are polynomials in $\uw$ and its conjugate. Note that $F(\lambda,\underline{\alpha})$ has a simple zero at $\lambda=1$ and this implies $\frac{\partial F}{\partial \lambda}(1,\underline{\alpha})\neq 0$. By implicit function theorem, we can find open neighborhoods $U_1$ of $1$ in $\mathbb{R}$ and $U_{\underline{\alpha}}$ of $\underline{\alpha}$ in $\mathbb{C}^3$, and a smooth function $g:U_{\underline{\alpha}}\rightarrow U_1$ such that 
		\begin{align*}
			\{(\lambda,\underline{w})\in U_1\times U_{\underline{\alpha}}:F(\lambda,\uw)=0\}=\{(g(\uw),\uw):\uw\in U_{\underline{\alpha}}\}.
		\end{align*}
		We can take $U_1\times U_{\underline{\alpha}}$ small enough so that $\frac{\partial F}{\partial \lambda}$ does not vanish in this neighborhood of $(1,\underline{\alpha})$. Thus for a fixed $\uw\in U_{\underline{\alpha}}$, $F(.,\uw)$ has exactly one zero in $U_1$, namely $g(\uw)$, and this zero depends smoothly on $\uw$.
		
		\textit{\textbf{Claim.}} We can choose $U_1\times U_{\underline{\alpha}}$ so that given $\uw\in U_{\underline{\alpha}}$, $g(\uw)$ is the largest zero of $F(.,\uw)$.
		
		\textit{\textbf{Proof of the claim.}} If not, then with a choice of $U_1\times U_{\underline{\alpha}}$ such that $\frac{\partial F}{\partial \lambda}$ does not vanish in this neighborhood, we can find a sequence $\{\uw _m\}$ in $U_{\underline{\alpha}}$ converging to $\underline{\alpha}$ and $\{\lambda_m\}$ in $\mathbb{R}$ such that $g(\uw _m)<\lambda_m$ and $F(\lambda_m, \uw_m)=0$. Since $|\lambda_m|$ are dominated by $||T_{\uw_m}||$ and the sequence $\{\uw_m\}$ converges, we may assume that $\{\lambda_m\}$ converges to some $\lambda\in \mathbb{R}$. Thus we get $F(\lambda, \underline{\alpha})=0$ and $\lambda\geq 1$. But this implies $\lambda=1$ and hence the sequence $\{\lambda_m\}$ is eventually in $U_1$. This contradicts the fact that $F(.,\uw_m)$ has exactly one zero in $U_1$ for all large enough $m$.

		Thus for all $\uw$ in a neighborhood $U_{\underline{\alpha}}$ of $\underline{\alpha}$, $g(\uw)=||\Tw^*||=||\Tw||$, and hence we have
		\begin{align}\label{gAndt}
			g(t\uw)=tg(\uw)
		\end{align} for all $t\in \mathbb{R}$ close to $1$. We differentiate (\ref{gAndt}) with with respect to $t$ and obtain 
		\begin{align*}
			1=g(\underline{\alpha})=\sum_{j=1}^{3}\Big\{\frac{\partial g}{\partial w_j} (\underline{\alpha})\alpha_j+\frac{\partial g}{\partial \overline{w_j}} (\underline{\alpha})\overline{\alpha_j}\Big\}.
		\end{align*}
		So the gradient of $g$ does not vanish near $\underline{\alpha}$ and hence $g$ gives a local defining function for $\partial\sDk$ at $\underline{\alpha}$. This implies that $\underline{\alpha}$ is not a singular point and we arrive at a contradiction. 
		
		Thus we have $rank\big(1-T_{\underline{\alpha}}^*T_{\underline{\alpha}}\big)=1$. By Lemma $3.1.$ in \cite{CW1996}, we can find $c_1,c_2,c_3\in \C$ such that $K_{i,j}=\frac{\overline{c_i}c_j}{1-\alpha_i\overline{\alpha_j}}$. Since $K$ is positive definite, we conclude that $\alpha_i$ are mutually distinct and $\sDk$ coincides with $\sD_\D (\underline{\alpha})$.
	\end{proof}
	
\begin{proof}[\textbf{Proof of Theorem \ref{3PointCharac}}]
	Without loss of generality, we may assume that $(\alpha_i,\alpha_j)\in \partial\sDo(z_i,z_j)$ for $(i,j)=(1,2)$ and $(1,3)$. Therefore we have $m(\alpha_1,\alpha_j)=\co(z_1,z_j),j=2,3$.
	
	We define $$K_{2\times 2} ^{(1,2)}=\begin{pmatrix}
		\frac{\sqrt{1-|\alpha_i|^2}\sqrt{1-|\alpha_j|^2}}{1-\alpha_i\overline{\alpha_j}}
	\end{pmatrix}_{1\leq i,j\leq 2}\,\, \text{and}\,\, K^{(1,2)}=\begin{pmatrix}
		K_{2\times 2} ^{(1,2)} &0\\
		0 & 1
	\end{pmatrix}.$$ Then $\sD_{K_{2\times 2} ^{(1,2)}}=\sD_\D (\alpha_1,\alpha_2)=\sDo(z_1,z_2)$ (\cite{Biswas2025}) and $\sD_{K^{(1,2)}}=\sDo(z_1,z_2)\times \overline{\D}$. Let $L^{(1,2)}$ be the supporting hyperplane to $\sDo(z_1,z_2)$ at $(\alpha_1,\alpha_2)$ in $\C^2$. We define $\Sigma^{(1,2)}=\{(w_1,w_2,w_3)\in \C^3:(w_1,w_2)\in L^{(1,2)}, w_3\in \C\}$ which is a supporting hyperplane to $\sD_{K^{(1,2)}}$ at $\underline{\alpha}$ in $\C^3$.
	
	Similarly we define $K_{2\times 2} ^{(1,3)}, K^{(1,3)}, L^{(1,3)}$ and the supporting hyperplane $\Sigma^{(1,3)}$ to $\sD_{K^{(1,3)}}$ at $\underline{\alpha}$ in $\C^3$.
	
	Note that $\underline{\alpha}\in \partial\sD_{K^{(1,2)}}\cap \partial\sD_{K^{(1,3)}}$. It is easy to see that the hyperplanes $\Sigma^{(1,2)}$ and $\Sigma^{(1,3)}$ are distinct (see page $210$ in \cite{CW1996}). Thus we have $\sDo(\uz)\subseteq \sD_{K^{(1,2)}}\cap \sD_{K^{(1,3)}}$ and $\sDo(\uz)$ lies in one of the four quadrants that $\Sigma^{(1,2)}$ and $\Sigma^{(1,3)}$ divides $\C^3$ into. This implies that $\underline{\alpha}$ is a singular point on $\partial\sDo(\uz)$, that is, on $\partial\sDk$. By Proposition \ref{SingularElement}, we conclude that $\alpha_1,\alpha_2$ and $\alpha_3$ are distinct and $\sDo(\uz)=\sD_\D (\underline{\alpha})$.
\end{proof}

Now we describe a few properties of extremal kernels.

\begin{lemma}\label{ExtremalKernel}
	A function $K:\{z_1,\ldots,z_n\}\times \{z_1,\ldots,z_n\}\rightarrow \mathbb{C}$ is  an extremal kernel if and only if $\pi_I(\sD_K (\underline{z}))=\sD_K(\underline{z}_I)$ for every $k\in \{1,\ldots,n\}$ and every $k$-tuple $I$.
\end{lemma}
\begin{proof}
	Let $K$ be an extremal kernel and consider any $k$-tuple $I=(i_1,\ldots,i_k)$. If $\underline{w}\in \sD_K(\underline{z})$, then there is a $\underline{w}_I\in \pi_I(\sD_K(\underline{z}))$ such that $\pi_I(\underline{w})=\underline{w}_I$. Since $\underline{w}\in \sD_K(\underline{z})$, we have $\big((1-w_i\overline{w_j})K(z_i,z_j)\big)\geq \mathbf{0}$ and this implies $\big((1-w_{i_l}\overline{w_{i_m}})K(z_{i_l},z_{i_m})\big)\geq \mathbf{0}$. Thus $\underline{w}_I\in \sD_K(\underline{z}_I)$ and hence $\pi_I(\sD_K (\underline{z}))\subseteq\sD_K(\underline{z}_I)$. Now for $\underline{x}\in \sD_K(\underline{z}_I)$ we find an $r\geq 1$ such that $r\underline{x}\in \partial \sD_K(\underline{z}_I)$. This gives us a $\underline{y}\in\partial\sD_K(\underline{z})$ such that $\pi_I(y)=r\underline{x}$. Now using the facts that $\pi_I$ is linear and $\frac{1}{r}\underline{y}\in \sD_K(\underline{z})$, we conclude that $x\in \pi_I(\sD_K(\underline{z}))$ and hence $\sD_K(\underline{z}_I)\subseteq \pi_I(\sD_K(\underline{z}))$.
	
	Conversely, let $\pi_I(\sDk(\uz))=\sDk(\uzi)$ for every $k$-tuple $I$. Let $\uw_I\in \partial \sDk(\uzi)$. We find a $\uw\in\sDk(\uz)$ such that $\uw_I=\pi_I(\uw)$. If $\uw$ lies in the interior of $\sDk(\uz)$, there is an $r>1$ such that $r\uw\in\sDk(\uz)$. This implies $\big((1-r^2w_l\overline{w_m})K(z_l,z_m)\big)_{l,m\in I}\geq \mathbf{0}$ which contradicts the assumption that $\uw_I$ is a boundary point. Thus $\uw\in \partial\sDk(\uz)$.
\end{proof}
	\begin{lemma}\label{ExtremalkernelEntry}
	Suppose that $z_1,\ldots,z_n$ are mutually distinct points in the Carath\'eodory hyperbolic domain $\Omega$ and $K$ is an $n$-extremal kernel on $\{z_1,\ldots,z_n\}$ with $K(z_i,z_i)=1$ for all $i$. If $\sDo(\uz)=\sDk(\uz)$, then $|K(z_i,z_j)|=\sqrt{1-\co(z_i,z_j)^2}$ for all $i$ and $j$. Also, $(0,\ldots,0,\mu)\in \sDk(\uz)$ if and only if $|\mu|\leq \co(z_n;z_1,\ldots,z_{n-1})$.
\end{lemma}
\begin{proof}
	For any pair $(i,j)$, we have, by Lemma \ref{ExtremalKernel}, $$\sDo(z_i,z_j)=\pi_{(i,j)}(\sDo(\uz))=\pi_{(i,j)}(\sDk(\uz))=\sDk(z_i,z_j).$$
	By Theorem $2.3$ in \cite{Biswas2025}, $\sDo(z_i,z_j)=\{(w_i,w_2)\in \D^2:m(w_1,z_2)\leq\co(z_i,z_j)\}\cup\{(e^{i\theta},e^{i\theta}):\theta\in \mathbb{R}\}$. On the other hand it is easy to see that $(w_1,w_2)\in \sDk(z_i,z_j)$ if and only if either $w_1=w_2\in \partial\D$ or $w_1,w_2\in \D$ and $|K(z_i,z_j)|\leq\sqrt{1-m(w_1,w_2)^2}$. Thus we obtain $|K(z_i,z_j)|\leq\sqrt{1-\co(z_i,z_j)^2}$. Since $\co(z_i,z_j)$ is attained, if we assume $|K(z_i,z_j)|<\sqrt{1-\co(z_i,z_j)^2}$, then there is an $r>1$ and $w_1,w_2\in \D$ such that $rw_1,rw_2\in \D$ and
	$$\co(z_i,z_j)=m(w_1,w_2)<m(rw_1,r w_2)<\sqrt{1-|K(z_i,z_j)|^2}.$$
	This implies $(rw_1,rw_2)\in \sDk(z_i,z_j)-\sDo(z_i,z_j)$ which is a contradiction.
	
	The statement involving $\mu$ follows easily from the definition of $\sDo(\uz)$. 
\end{proof}

\begin{proof}[\textbf{Proof of Theorem \ref{BoundOnGeneralizedCara}}]
	Without loss of generality, we may assume that $K(z_i,z_i)=1$ for all $i$. By Lemma \ref{ExtremalkernelEntry}, there is a $\theta\in \mathbb{R}$ such that $K(z_1,z_2)=e^{i\theta}\frac{\sqrt{1-|\alpha_1|^2}\sqrt{1-|\alpha_2|^2}}{1-\alpha_1\overline{\alpha_2}}$. We find a nonzero $(v_1,v_2)\in \C^2$ such that 
	\begin{align}\label{ZeroOf2x2}
		\begin{pmatrix}
			(1-|\alpha_1|^2) & (1-\alpha_1\overline{\alpha_2})K(z_1,z_2)  \\
			(1-\alpha_2 \overline{\alpha_1})K(z_2,z_1)  & (1-|\alpha_2|^2) 
		\end{pmatrix}\begin{pmatrix}
			v_1  \\
			v_2
		\end{pmatrix}=\mathbf{0}.\end{align}
	
	Note that, $v_1v_2\neq 0$. Now, for any $t,v_3\in \mathbb{C}$, we now have that the quantity
	\begin{align*}
		\Bigg\langle\begin{pmatrix}
			(1-|\alpha_1|^2) & (1-\alpha_1\overline{\alpha_2})K(z_1,z_2)& (1-\alpha_1\overline{\alpha_3})K(z_1,z_3)   \\
			(1-\alpha_2 \overline{\alpha_1})K(z_2,z_1)  & (1-|\alpha_2|^2) &(1-\alpha_2\overline{\alpha_3})K(z_2,z_3) \\
			(1-\alpha_3\overline{\alpha_1})K(z_3,z_1) &(1-\alpha_3\overline{\alpha_2})K(z_3,z_2) &(1-|\alpha_3|^2)
		\end{pmatrix}\begin{pmatrix}
			tv_1  \\
			tv_2\\
			v_3
		\end{pmatrix},\begin{pmatrix}
			tv_1  \\
			tv_2\\
			v_3
		\end{pmatrix}\Bigg\rangle
	\end{align*}
	is non-negative. Simplifying it and using (\ref{ZeroOf2x2}) we find
	$$(1-|\alpha_3|^2)|v_3|^2+2Re\bigg(t\overline{v_3}\big((1-\alpha_3\overline{\alpha_1})K(z_3,z_1)v_1 +(1-\alpha_3\overline{\alpha_2})K(z_3,z_2)v_2\big)\bigg)\geq 0$$
	for all $t,v_3\in \C$. Hence we conclude $(1-\alpha_3\overline{\alpha_1})K(z_3,z_1)v_1 +(1-\alpha_3\overline{\alpha_2})K(z_3,z_2)v_2=0$. Again using (\ref{ZeroOf2x2}), we see that there is a nonzero $\mu\in \C$ such that 
	\begin{align*}\label{OtherValuesOfK}
		K(z_1,z_2)=e^{i\theta}M_{1,2},K(z_3,z_1)=\mu M_{3,1},\,\,\text{and}\,\,K(z_3,z_2)=\mu e^{i\theta}M_{3,2}
	\end{align*}
	where $M_{i,j}=\frac{\sqrt{1-|\alpha_i|^2}\sqrt{1-|\alpha_j|^2}}{1-\alpha_i\overline{\alpha_j}}$. Therefore we have
	\begin{align*}
		|\mu|=\frac{\sqrt{1-\co(z_j,z_3)^2}}{\sqrt{1-m(\alpha_j,\alpha_3)^2}}\,\,\,\text{for}\,\,\,\,j=1,2.
	\end{align*}
	Now it is easy to see that $\Bigg(0,0,\sqrt{\frac{\co(z_j,z_3)^2 -m(\alpha_j,\alpha_3)^2}{1-m(\alpha_j,\alpha_3)^2}}\Bigg)\in \sDk(\uz)$ and the result follows from Lemma \ref{ExtremalkernelEntry}.
\end{proof}

\begin{lemma}\label{SzegoExtremal}
	Let $K$ be an $n$-extremal kernel defined on $\{1,\ldots, n\}$ with $K(i,i)=1$ for all $i$ and let $\alpha_1,\ldots,\alpha_n$ be points in $\D$. If $\sDk=\sD_\D (\underline{\alpha})$, then $\alpha_1,\ldots,\alpha_n$ are mutually distinct and there exist $\theta_1,\ldots,\theta_n\in \mathbb{R}$ such that
	
	\begin{align*}
		K(l,m)=e^{i(\theta_l-\theta_m)}\frac{\sqrt{1-|\alpha_l|^2}\sqrt{1-|\alpha_m|^2}}{1-\alpha_l\overline{\alpha_m}}
	\end{align*}
	for all $l$ and $m$.
\end{lemma}
\begin{proof}
	Since $K$ is positive definite, we can find a $(w_1,\ldots,w_n)\in \sDk(1,\ldots,n)$ such that $w_i\neq w_j$ for all $i\neq j$. So 
	
	\begin{align*}
		\begin{pmatrix}
			\frac{1-w_i\overline{w_j}}{1-\alpha_i\overline{\alpha_j}}
		\end{pmatrix}\geq \mathbf{0},
	\end{align*}
	and using Lemma $3.2$ in \cite{CW1996} we find that $\alpha_1,\ldots,\alpha_n$ are mutually distinct.
	
	By Lemma \ref{ExtremalkernelEntry}, $|K(l,m)|=\sqrt{1-m(\alpha_l,\alpha_m)^2}$ for all $l$ and $m$, and hence there exist $\theta_{l,m}\in \mathbb{R}$ such that $K(l,m)=e^{i\theta_{l,m}}\frac{\sqrt{1-|\alpha_l|^2}\sqrt{1-|\alpha_m|^2}}{1-\alpha_l\overline{\alpha_m}}$ with $\theta_{l,l}=0$ for all $l$. Now, because $\underline{\alpha}\in\sDk(1,\ldots,n)$, we have 
	$\Big(\frac{1-\alpha_l\overline{\alpha_m}}{\sqrt{1-|\alpha_l|^2}\sqrt{1-|\alpha_m|^2}}K(l,m)\Big)=\big(e^{i\theta_{l,m}}\big)\geq \mathbf{0}$. Since $\theta_{l,l}=0$ for all $l$, the result follows from this.
\end{proof}

\begin{proof}[\textbf{Proof of Theorem \ref{nPointTheorem}}]
	We apply induction on $n$. For $n=2$, we have $\sDo(z_1,z_2)=\sD_\D (\alpha_1,\alpha_2)$ by Theorem 2.3 in \cite{Biswas2025}. Let the statement be true for $n-1$.
	
	Suppose now we are given an $n$-extremal kernel $K$ on $\{z_1,\ldots,z_n\}$ satisfying $\sDk(\uz)=\sDo(\uz)$ and a point $\underline{\alpha}=(\alpha_1,\ldots,\alpha_n)\in \sDo(\uz)$ such that $m(\alpha_1,\alpha_j)=\co(z_1,z_j)$ for $j=2,\ldots, n$. Since $K$ is extremal, we have $\sDk(z_1,\ldots,z_{n-1})=\sDo(z_1,\ldots,z_{n-1})$, and the point $(\alpha_1,\ldots,\alpha_{n-1})\in \sDo(z_1,\ldots,z_{n-1})$ satisfies $m(\alpha_1,\alpha_j)=\co(z_1,z_j)$ for $j=2,\ldots, n-1$. By induction hypothesis and Lemma \ref{SzegoExtremal}, there are $\theta_1,\ldots,\theta_{n-1}\in \mathbb{R}$ such that 
	\begin{align*}
		K(z_l,z_m)=e^{i(\theta_l -\theta_m)}\frac{\sqrt{1-|\alpha_l|^2}\sqrt{1-|\alpha_m|^2}}{1-\alpha_l\overline{\alpha_m}}
	\end{align*}
	for all $l,m=1\ldots,n-1$. Therefore, we have
	\begin{align*}
		\begin{pmatrix}
			\frac{1-\alpha_i\overline{\alpha_j}}{\sqrt{1-|\alpha_i|^2}\sqrt{1-|\alpha_j|^2}}K(z_i,z_j)
		\end{pmatrix}_{n\times n}=\begin{pmatrix}
			\big(e^{i(\theta_l -\theta_m)}\big)_{n-1 \times n-1} & \big(a_{j,n}\big)_{n-1\times 1}\\
			\big(a_{n,j}\big)_{1\times n-1} &1
		\end{pmatrix}\geq \mathbf{0},
	\end{align*}
	where $a_{j,n}=\frac{1-\alpha_j\overline{\alpha_n}}{\sqrt{1-|\alpha_j|^2}\sqrt{1-|\alpha_n|^2}}K(z_j,z_n)$. Note that $|a_{1,n}|=1$ by Lemma \ref{ExtremalkernelEntry}. We now consider any $l,m\in \{1,\ldots, n-1\}$ with $l<m$ and the $\{l,m,n\}\times \{l,m,n\}$ block, that is,
	\begin{align*}
		A_{l,m}=\begin{pmatrix}
			1& e^{i(\theta_l -\theta_m)} & a_{l,n}\\
			e^{i(\theta_m -\theta_l)} & 1&a_{m,n}\\
			a_{n,l} & a_{n,m}&1
		\end{pmatrix}.
	\end{align*}
	Since $A_{l,m}$ is positive semidefinite, for any $t,v\in \C$
	\begin{align*}
		\Bigg\langle A_{l,m}\begin{pmatrix}
			-t\\
			te^{i(\theta_m -\theta_l)}\\
			v
		\end{pmatrix},\begin{pmatrix}
			-t\\
			te^{i(\theta_m -\theta_l)}\\
			v
		\end{pmatrix}\Bigg\rangle \geq 0.
	\end{align*}
	Hence $|v|^2 +2Re\big(v\overline{t}(-a_{n,l}+e^{i(\theta_m -\theta_l)}a_{n.m})\big)\geq 0$ for all $t,v\in \C$. Thus we find $e^{i\theta_l}a_{n,l}=e^{i\theta_m}a_{n,m}$. Since $l$ and $m$ is arbitrary and $|a_{1,n}|=1$, there is a $\theta_n\in \mathbb{R}$ such that $a_{l,n}=e^{i(\theta_l -\theta_n)}$ for all $l$. This is sufficient to conclude $\sDk(\uz)=\sD_\D (\underline{\alpha})$, that is, $\sDo(\uz)=\sD_\D (\underline{\alpha})$.
\end{proof}

\vspace{0.1in} \noindent\textbf{Acknowledgements:} The work is supported by European Horizon MSCA grant,\\ CZ.02.01.01/00/22 10/0008854.

\end{document}